\documentclass{amsart}
\usepackage{amssymb, amsfonts, amsmath, amsthm}

\DeclareMathSymbol{\twoheadrightarrow}  {\mathrel}{AMSa}{"10}

\DeclareMathSymbol{\twoheadrightarrow} {\mathrel}{AMSa}{"10}

\def\C{{\mathbb C}}

\def\A{{\mathbf A}}

              \def\Mult{\mathrm{Mult}}

        \def\K_a{\bar{K}}

\def\M{\mathrm{M}}
                           \def\N{\mathrm{N}}

            \def\grad{\mathrm{grad}}

\def\K{{\mathcal{K}}}

\def\Oc{{\mathcal O}}

\newtheorem{thm}{Theorem}[section]

\newtheorem{lem}[thm]{Lemma}
\newtheorem{cor}[thm]{Corollary}

\theoremstyle{definition}

\newtheorem{ex}[thm]{Example}
\newtheorem{rem}[thm]{Remark}

        \newtheorem{sect}[thm]{}

\title[One-dimensional polynomial maps]{One-dimensional polynomial maps, periodic points and multipliers}

\author{Yuri G. Zarhin (Zarkhin)}
\thanks{This work was partially supported by a grant from the Simons Foundation (\#246625 to Yuri Zarkhin).}

\address{Department of Mathematics, Pennsylvania
State University, University Park, PA 16802, USA}

\address{Institute of Mathematical Problems of Biology,
Russian Academy of Sciences, Pushchino, Moscow Region, Russia}

\email{zarhin\char`\@math.psu.edu}

\begin{document}
\begin{abstract}
We discuss tangent maps related to the multipliers of periodic
points of a typical one-dimensional polynomial map.

UDC 517.535.2, 517.927.7
\end{abstract}

 \maketitle

\section{Definitions, Notation, Statements}
We write $\C$ for the field of complex numbers. For every positive
integer $m$ let us  consider the affine space $\A^m=\C^m$ of all
monic complex polynomials of degree  $m$
$$u(x)=x^m+\sum_{i=0}^{m-1} a_i x^i$$ with coefficients
 $a=(a_0, \dots , a_{m-1})\in \C^m=\A^m$.
It is convenient to identify the tangent
  space
   $\C^m$  to $u(x)\in \A^m$ with the space of all polynomials
  $p(x)$ of degree $\le m-1$. Namely, to a polynomial $p(x)=\sum_{i=0}^{m-1} c_i x^i$
 one assigns the tangent vector $(c_0, \dots, c_{m-1}) \in \C^m$
 that corresponds to
``the tangency class  at $u(x)$ of the curve"
 $\epsilon \to u(x)+\epsilon \cdot p(x) \in \A^m$ \cite[Part II, Ch. III, Sect. 8, pp. 81--82]{Serre}.

Let $P_m\subset\A^m$ be the everywhere dense Zariski-open affine
subset
  that consists of all  polynomials  without multiple roots. Let $f(x)=x^m+\sum_{i=0}^{m-1} a_i x^i\in P_m$
and let us choose a root $\alpha$   of $f(x)$. Locally (with respect
to $a$), one may view
  $\alpha$  (using Implicit Function Theorem) as a holomorphic
(univalued) function in  $a=(a_0, \dots , a_{m-1})$.  We have
(\cite[Sect. 2]{ZarhinMatZametki2012})
$$d\alpha/da_i=-  [f^{\prime}(\alpha)]^{-1} \alpha^i.$$
(Since $\alpha$ is a simple root of $f(x)$, we have $f^{\prime}(\alpha)\ne 0$.)
We also have (ibid)
 $$d f^{\prime}(\alpha)/d a_i =i \alpha^{i-1} - [f^{\prime}(\alpha)]^{-1} \alpha^i  f^{\prime\prime}(\alpha)$$
 (of course, if $i=0$ then the first term disappears). Using these formulas, let us compute the differential
 $d\N: \C^m \to \C$ (at $f(x)$) of   locally defined holomorphic function
$$\N: P_{m} \to  \C, \ f(x)  \mapsto f^{\prime}(\alpha).$$

 It follows that $d\N$ sends the tangent vector $p(x)=\sum_{i=0}^{m-1} c_i x^i$ to the number
$$d\N(p(x))=\sum_{i=0}^{m-1} c_i \frac{d f^{\prime}}{d a_i}(\alpha)=p^{\prime}(\alpha)- [f^{\prime}(\alpha)]^{-1}p(\alpha)  f^{\prime\prime}(\alpha).$$

\begin{ex}
\label{exm}
Suppose that $m \ge 3$ and $f(x)=x^m-x$. Then $\alpha$ is either zero or $(m-1)$th  root of unity.
If $\alpha=0$ then $f^{\prime\prime}(0)=0$ and
$$d\N(p(x))=p^{\prime}(0)=c_1.$$
The {\sl gradient} of $\N$ at $f(x)=x^m-x$ (with respect to the root
$0$) is
$$Q_1(0)=(0,1, \dots ,0) \in \C^m.$$
If $\alpha^{m-1}=1$ then
$$f^{\prime}(\alpha)=m\alpha^{m-1}-1=m-1,$$
$$f^{\prime\prime}(\alpha)=m(m-1)\alpha^{m-2}=m(m-1)/\alpha,$$
and
$$d\N(p(x))=p^{\prime}(\alpha)-\frac{m p(\alpha)}{\alpha}.$$
The {\sl gradient}  of $\N$ at $f(x)=x^m-x$ (with respect to the
root $\alpha$) is
$$Q_1(\alpha)=\left(-\frac{m}{\alpha}, (1-m), (2-m)\alpha, \dots ,-\alpha^{m-2}\right) \in \C^n.$$
\end{ex}

Let $n \ge 2$ be an integer and $g(x) \in \C[x]$ a degree $n$ monic
polynomial with complex coefficients. For every positive integer $r$
we denote by $g^{\circ r}(x)$ the composition $g(\dots g(x))$ ($r$
times). Clearly, $g^{\circ r}(x)$ is a degree $n^r$ monic polynomial
with complex coefficients.
Let us consider the polynomial map
$$G: \C \to\C, \ z \mapsto g(z).$$
Clearly, the fixed points of $G$ are exactly the roots of $g(x)-x$
while the roots of  $g^{\circ r}(x)-x$ are exactly the points of
period (dividing) $r$.

\begin{ex}
If $g(x)=x^n$ then $g^{\circ r}(x)=x^{n^r}, \ g^{\circ r}(x)-x=
x^{n^r}-x$.
\end{ex}

 We write $Z_{n,r}\subset
\A^n$ for the everywhere dense Zariski-open  affine subset that
consists of all monic degree $n$ polynomials $g(x)$ such that
$g^{\circ r}(x)-x$ lies in $P_{n^r}$ (i.e., does not have multiple
roots). For example, $x^n \in Z_{n,r}$ for all $r$. Clearly, for
every positive integer $m$
$$Z_{m,1}= \{f(x)+x\mid f(x)\in P_m \}.$$
It is also clear that the holomorphic map
$$U_m: Z_{m,1} \to P_{m}, \ h(x)\mapsto h(x)-x$$
is a holomorphic isomorphism, whose tangent map
$$dU_n:\C^m \to \C^m$$
is the identity map at all points of $Z_{m,1}$.

 Let us consider a locally defined holomorphic function
$$\M^{r}: Z_{n,r}\to Z_{n^r,1} \to P_{n^r} \to  \C, \ g(x) \mapsto g^{\circ r}(x)
\overset{U_{n^r}}{\longmapsto}
  g^{\circ r}(x)-x \overset{\N}{\mapsto} [g^{\circ
r}(x)-x]^{\prime}(\alpha)$$ where $\alpha$ is a root of $g^{\circ
r}(x)-x$. We are going to discuss its differential $$d\M^{r}_{\mid
g(x)}:\C^n \to \C,$$ paying special attention to the computation of
the corresponding gradient
$$\grad(\M^{r})_{\mid g(x)}\in \C^n$$
 at the point $g(x)=x^n\in Z_{n,r}$. In what follows we denote this gradient by $Q_r(\alpha)$.
This notation is compatible with our previous notation for $Q_1(\alpha)$ in Example \ref{exm}.

\begin{rem}
Let us consider the locally defined multiplier function
$$\Mult^{r}: Z_{n,r}\to Z_{n^r,1} \to  \C, \ g(x) \longmapsto [g^{\circ r}]^{\prime}(\alpha).$$
Clearly, $\M^{r}(g)=\Mult^{r}(g)- 1$. It follows that the differentials
 $d\M^{r}$ and $d\Mult^{r}$ everywhere coincide. In other words
$$\grad(\M^{r} )_{\mid g(x)}=\grad(\Mult^{r})_{\mid g(x)} \ \forall g(x) \in Z_{n,r}.$$
\end{rem}

\begin{sect}
In order to state our main results, first notice that if $g(x) \in
\in Z_{n,r}$ and $\nu(n,r)$ is the number of of all orbits of length
$r$ for the map $z \mapsto g(z)$ then
$$\nu(n,r) \ge \frac{n}{r}$$
(see Subsection \ref{numberoforbits}). Second, let us consider a
positive integer $\ell$ and a sequence $\{r_1, \dots, r_{\ell}\}$ of
$\ell$ positive integers. Let $Z(n,\ell; r_1, \dots , r_{\ell})$ be
the intersection of all $Z_{n,r_i}$; it is a nonempty Zariski-open
affine subset in $\A^n$ that contains $g(x)=x^n$. Let $g(x)\in
Z(n,\ell; r_1, \dots , r_{\ell})$.
 For each $i$ pick a complex number $\beta_i$ that is a periodic point of $G: z\mapsto g(z)$
of {\bf exact period} $r_i$. Locally (with respect to  $g$), each
$\beta_i$ is a holomorphic function (in the coefficients of $g(x)$.)

Suppose that $\beta_1,  \dots \beta_{\ell}$ belong to {\bf distinct orbits} of $z\mapsto g(z)$. Let us consider the following $\ell$
locally defined holomorphic functions
$$\Mult_{\beta_i,r_i}: Z(n,\ell; r_1, \dots , r_{\ell})\to \C, \ g(x) \mapsto  [g^{\circ r_i}]^{\prime}(\beta_i).$$
Let  $Z^0(n,\ell; r_1, \dots , r_{\ell})$ be the set of all polynomials $g(x) \in Z(n,\ell; r_1, \dots , r_{\ell})$
such that the $\ell$-element set
$$\{\grad(\Mult_{\beta_i,r_i})_{\mid g(x)}\in \C^n \ \mid 1 \le i\le \ell\}$$
of gradients of $\Mult_{\beta_i,r_i}$'s at $g(x)$ is linearly
independent in $\C^n$  for every choice of $\{\beta_1, \dots,
\beta_{\ell}\}$. Clearly, $Z^0(n,\ell; r_1, \dots , r_{\ell})$  is
an open subset of $Z(n,\ell; r_1, \dots , r_{\ell})$ and therefore
of $\A^n$ in complex topology. However, this set may be empty; e.g.,
when $\ell\ge n$.
\end{sect}

The following statements are  main results of this paper.

\begin{thm}
 \label{open}
The set  $Z^0(n,\ell; r_1, \dots , r_{\ell})$  is  a Zariski-open subset of $Z(n,\ell; r_1, \dots , r_{\ell})$ and
therefore of $\A^n$.
\end{thm}

\begin{thm}
 \label{mainT}
Suppose that $n \ge 3$. Assume that $\sum_{i=1}^{\ell}r_i \le n$. If
$r_j=1$ for some $j$ then we assume additionally that
$\sum_{i=1}^{\ell}r_i < n$.

Then    $Z^0(n,\ell; r_1, \dots , r_{\ell})$ contains $g(x)=x^n$  and therefore is nonempty.
\end{thm}

Notice that under the notation and assumptions of Theorem
\ref{mainT}, if $r$ is a positive integer and $l(r)$ is the number
of $i$'s with $r_i=r$ then
$$l(r) \le \frac{n}{r} \le \nu(n,r).$$

Combining Theorems \ref{open} and \ref{mainT}, we obtain the
following statement.

\begin{cor}
\label{maincor} Suppose that $n \ge 3$. Assume that
$\sum_{i=1}^{\ell}r_i \le n$. If $r_j=1$ for some $j$ then we assume
additionally that $\sum_{i=1}^{\ell}r_i < n$.

Then $Z^0(n,\ell; r_1, \dots , r_{\ell})$ is a Zariski-open
everywhere dense  subset of $Z(n,\ell; r_1, \dots , r_{\ell})$ that
contains $g(x)=x^n$.
\end{cor}

\begin{ex}
 Suppose that $\ell = n-1$ and all $r_i=1$ (i.e., all the $\beta_i$ involved are  fixed points).
 It follows from results of \cite{ZarhinMatZametki2012,Rees}
that $$Z^0(n,n-1;1, \dots , 1)=Z(n,n-1;1, \dots ,1)=Z_{n,1}.$$
\end{ex}

\begin{rem}
\label{maprank}
 In the notation and assumptions of Corollary \ref{maincor}, let us consider the locally defined holomorphic map
$$Z(n,\ell; r_1, \dots , r_{\ell})\to \C^{\ell}$$
defined by the collection of functions
$\{\Mult_{\beta_i,r_i}\}_{i=1}^{\ell}$. Corollary \ref{maincor}
 asserts that this map has (maximal) rank $\ell$ on a nonempty Zariski-open
 subset
  $$Z^0(n,\ell; r_1, \dots , r_{\ell}) \subset Z(n,\ell; r_1, \dots , r_{\ell})$$
for every choice of periodic points $\{\beta_1, \dots ,
\beta_{\ell}\}$. It would be interesting to study its image. For the
case of fixed points (i.e., when all $r_i=1$), see \cite{Kulikov}.
\end{rem}

Notice that Remark \ref{maprank} gives a partial answer to a
question of Yu.S. Ilyashenko, who was interested in the case of two
orbits, in connection with \cite{Ilyash,Igor}.

\begin{rem}
In the case of two orbits,
 it turns out (see Examples \ref{n3r2}, \ref{rplus1} and Remark \ref{rplus2} below) that $g(x)=x^n$ does not belong
to $Z(n,2; r_1,r_2)$ if either $r_1=r_2=n-1$ or $r_1=1, r_2=n-1$. It would be interesting to find out whether in
these cases $Z(n,2; r_1,r_2)$ is empty.
\end{rem}

The paper is organized as follows.
In Section \ref{compute} we compute explicitly the differentials
$d\Mult_{\beta,r}$ at $g(x)=x^n$
where $\beta$ is a $(n^r-1)$th root of unity.
 This allows us to write down explicitly
the corresponding gradients $Q_r(\beta)\in\C^n$.
Now Theorem \ref{mainT} becomes equivalent to an assertion that
the corresponding set of vectors $\{Q_{r_i}(\beta_i)\}$ (and $Q_1(0)$ if one of $r_i$ is $1$)
 is linearly independent in $\C^n$.
We prove this assertion in Section \ref{ind}. Using standard
properties of finite maps (\cite[Ch. 1]{Sh},  \cite[Sect.
8]{Milne}), we prove Theorem \ref{open}  in Section \ref{global}.

\section{Computations of tangent maps}
\label{compute}

\begin{lem}
\label{tangentiter} Let us consider the holomorphic map $\Phi_{n,r}:
\A^n \to \A^{n^r}$ that sends a degree $n$ monic polynomial $g(x)$ to
the monic degree $n^r$ polynomial $g^{\circ r}(x)$. Then the tangent
map $d\Phi_{n,r}$ at $g(x)=x^n$ is as follows. It sends a tangent
vector $x^k$ (at the point $x^n$) to the tangent vector
$$p_{r,k}(x):=\sum_{i=1}^r n^{r-i} x^{n^r-n^i+n^{i-1}k}$$
(at the point $x^{n^r}$).
 In particular,
$$p_{r,0}(x):=\sum_{i=1}^r n^{r-i} x^{n^r-n^i}=n^{r-1}x^{n^r-n}+n^{r-2}x^{n^r-n^2}+\dots ,$$
$$p_{r,1}(x):=\sum_{i=1}^r n^{r-i} x^{n^r-n^i+n^{i-1}}=n^{r-1}x^{n^r-n+1}+n^{r-2}x^{n^r-n^2+n}+\dots $$
and $$\deg(p_{r,0})=n^r-n, \ \deg(p_{r,1})=n^r-n+1, \
\deg(p_{r,k})=n^r-n+k.$$
\end{lem}

\begin{proof}
Notice that $p_{1,k}(x)=x^k$ and for all positive integers $r$
$$p_{r+1,k}(x)=n x^{(n-1)n^r}p_{r,k}(x)+x^{k
n^r}.$$

{\sl Induction by} $r$.
 We need to prove that if $g^{[\epsilon]}(x)=x^n+\epsilon x^k$
then $[g^{[\epsilon]}]^{\circ r}(x)=x^{n^k}+\epsilon
p_{r,k}(x)+O(\epsilon^2)$. If $r=1$ then it is obvious. Assume that
this assertion is true for $r$ and let us check it for $r+1$. We
have
$$[g^{[\epsilon]}]^{\circ (r+1)}(x)=$$
$$\{[g^{[\epsilon]}]^{\circ r}(x)\}^n+\epsilon
\{[g^{[\epsilon]}]^{\circ r}(x)\}^k= (x^{n^r}+\epsilon
p_{r,k}(x)+O(\epsilon^2))^n+\epsilon (x^{n^r}+\epsilon
p_{r,k}(x)+O(\epsilon^2))^k=$$
$$x^{n^{r+1}}+\epsilon n
x^{(n-1)n^r}p_{r,k}(x)+\epsilon x^{k n^r}+O(\epsilon^2)=$$
$$x^{n^{r+1}}+\epsilon\{n x^{(n-1)n^r}p_{r,k}(x)+x^{k
n^r}\}+O(\epsilon^2)= x^{n^{r+1}}+\epsilon
p_{r,k+1}(x)+O(\epsilon^2).$$

\end{proof}

 Since all $p_{r,k}$ (for given $n$ and $r$) have distinct degrees, the
set $\{p_{r,0}, \dots , p_{r,n-1}\}$ is linearly independent. This
means that the rank of the tangent map to $\Phi_{n,r}$ at $g(x)=x^n$
is $n$, i.e. the tangent map at this point is injective and its
image coincides with
$$\oplus_{k=0}^{n-1}\C\cdot p_{r,k}.$$

\begin{sect}
\label{tangentMult} Suppose that $n^r \ge 3$.
Let us compute the differential
 $$d\Mult^{r}=d\M^{r}=d(\M \Phi_{n,r})=d\M \circ d \Phi_{n,r}$$
  at $g(x)=x^n \in Z_{n,r}$. Clearly,
  $$\Phi_{n,r}(x^n)=x^{n^r}\in P_m$$
  with $m=n^r$. Let $\alpha$ be a nonzero root of $x^m-x$, i.e., $\alpha^{n^r-1}=1$. Using Lemma \ref{tangentiter} and Example \ref{exm}, we obtain the following.
  The image
  $$q_{r,k}(\alpha):=d\Mult^{r}_{\mid g(x)=x^n}(x^k)$$
   of tangent vector $x^k$ to $g(x)=x^n \in Z_{n,r}$ is
  $$p_{r,k}^{\prime}(\alpha)-\frac{n^r p_{r,k}(\alpha)}{\alpha}=$$
 $$\alpha^{-1}\sum_{i=1}^r  (n^r-n^i+n^{i-1} k) n^{r-i}
\alpha^{n^r-n^i+n^{i-1}k}-\alpha^{-1}\sum_{i=1}^r n^r  n^{r-i}
\alpha^{n^r-n^i+n^{i-1}k}=$$
 $$\alpha^{n^r-1}\sum_{i=1}^r [n^{2r-i}-n^r+k
n^{r-1}-n^{2r-i}]\alpha^{-n^i+n^{i-1}k}=-(n^r-k n^{r-1})\sum_{i=1}^r
\alpha^{-n^i+n^{i-1}k}=$$
$$-(n-k)n^{r-1}\sum_{i=1}^r
\alpha^{-n^i+n^{i-1}k}=
-(n-k)n^{r-1}\sum_{i=1}^r \left(\frac{1}{\alpha^{n^{i-1}}}\right)^{n-k}.$$
In other words,
$$q_{r,k}(\alpha)=d\Mult^{r}_{\mid g(x)=x^n}(x^k)=-(n-k)n^{r-1}\sum_{i=1}^r \left(1/\alpha^{n^{i-1}}\right)^{n-k}.$$
Notice that
$$q_{r,k}(\alpha)=q_{r,k}(\alpha^n), \ q_{r,0}(\alpha)=n\cdot q_{r,n-1}(\alpha).$$
It follows that the {\sl gradient} of $\Mult^{r}$ at $g(x)=x^n$
(with respect  to $\alpha$) is
$$Q_{r}(\alpha)=(q_{r,0}(\alpha), q_{r,1}(\alpha), \dots, q_{r,n-1}(\alpha))=(n q_{r,n-1}(\alpha), q_{r,1}(\alpha), \dots, q_{r,n-1}(\alpha))\in \C^n.$$
Clearly,
$$Q_{r}(\alpha)=Q_{r}(\alpha^n)= \dots = Q_{r}(\alpha^{n^{r-1}}).$$

Let $\Oc(\alpha)=\{\alpha, \alpha^n, \dots , \alpha^{n^{r-1}}\}$ be the orbit of $\alpha$ with respect to $z\mapsto z^n$ and let $d(\alpha)$ be the cardinality of the set $\Oc(\alpha)$. Clearly, $d(\alpha)$ is a positive integer that divides $r$ and
$$\beta^{d(\alpha)}=\beta \ \forall \beta \in \Oc(\alpha).$$
It is also clear that
$$q_{r,k}(\alpha)=-\frac{r}{d(\alpha)}(n-k)n^{r-1}\sum_{\beta \in \Oc(\alpha)} \left(1/\beta\right)^{n-k}.$$
This implies that
$$Q_r(\alpha) =\frac{rn^{r}}{d(\alpha)n^{d(\alpha)}}\cdot Q_{d(\alpha)}(\alpha).$$
 \end{sect}

\begin{ex}
 \label{n3r2}
Suppose that $n=3$ and $r=2$. Then $n^r-1=8$. Let $\alpha$ be a $8$th root of unity that is not $\pm 1$.
Then $\alpha$ is a periodic point of exact period $2$ for $z\mapsto z^3$. We have
$$Q_2(\alpha)=
 -3^1\cdot \left(3\left[\frac{1}{\alpha}+\frac{1}{\alpha^3}\right],2\left[\frac{1}{\alpha^2}+\frac{1}{\alpha^6}\right],
      \left[\frac{1}{\alpha}+\frac{1}{\alpha^3}\right]\right).$$
So, if $\alpha$ is a primitive fourth root of unity, i.e.,
$$\alpha^2=-1, \ \alpha=\pm \mathbf{i},$$
then
$$\frac{1}{\alpha}+\frac{1}{\alpha^3}=0, \ \frac{1}{\alpha^2}+\frac{1}{\alpha^6}=-2$$ and
       $$Q_2(\alpha)=-3\cdot (0, -2,0)=(0,6,0).$$
 (Notice that $\mathbf{i}$ and $-\mathbf{i}$ lie in the same orbit.)

If $\alpha$ is a primitive $8$th root of unity then
$$\alpha^4=-1, \ 1+\frac{1}{\alpha^4}=0$$ and therefore
$$Q_2(\alpha)=
   -3\cdot \left(3\left[\frac{1}{\alpha}+\frac{1}{\alpha^3}\right],0,
      \left[\frac{1}{\alpha}+\frac{1}{\alpha^3}\right]\right)=
       -\frac{3}{\alpha^3}\cdot \left(3[\alpha^2+1],0, \alpha^2+1)\right)=
-\frac{3(\alpha^2+1)}{\alpha^3}\cdot \big(3,0,1\big).$$
 Now if we put $\beta=\alpha^{-1}$ then  $\alpha$ and $\beta$ are primitive $8$th roots of unity that do not belong
to the same orbit while $Q_2(\alpha)$ and $Q_2(\beta)$ generate the same line $\C\cdot (3,0,1)$ in $\C^3$. This implies
that $Z(3,2;2,2)$ does {\sl not} contain $g(x)=x^3$.

\end{ex}

\begin{ex}
\label{rplus1} Suppose that $n=r+1$ and $r\ge 2$. Then for all
positive integers $i$
$$n^i=(1+r)^i=1+ i \cdot r^{1}+ \dots + {i\choose j}r^{j} + \dots +r^i.$$
It follows that $n^i$ is congruent to $1+ir$ modulo $r^2$. In
particular, $n^r-1$ is divisible by $r^2=(n-1)^2$. It also follows
that   $(n^i-n)/(n-1)=n^{i-1}$ is congruent to $i-1$ modulo $r$ and
therefore
$$n^i-n \equiv (i-1)r \bmod r^2.$$

Suppose that $\alpha$ is a {\sl primitive} $r^2$th root of unity. Then $\alpha^{n^r-1}=1$ and therefore
$$\alpha^{n^r}=\alpha,$$
i.e., $\alpha$ is a periodic point for the map $z\mapsto z^n$. Clearly, its period divides $r$. On the other hand,
 for all positive integers $i<r$ the power $n^i$ is {\sl not} congruent to $1$ modulo $r^2$ and
therefore $\alpha^{n^i}\ne \alpha$. It follows that $\alpha$ has
exact period $r$.

The number $\gamma:=\alpha^{1-n}=\alpha^{-r}$ is a primitive $r$th
root of unity. For each integer $k$ with $0\le k\le n-1$ the number
$\delta=\gamma^{n-k}$ is an $r$th root of unity. Clearly, $\delta
\ne 1$ if and only if $n-k \ne n-1$, i.e. $k \ne 1$.  In particular,
if $k \ne 1$ then $\sum_{i=1}^r \delta^i=0$.

We have
$$q_{r,k}(\alpha)=-(n-k)n^{r-1} \sum_{i=1}^r
\left(\alpha^{-n^i}\right)^{n-k}=$$
$$-(n-k)n^{r-1} \cdot\alpha^{n(k-n)}\sum_{i=1}^r
\left(\alpha^{n-n^i}\right)^{n-k}=-(n-k)n^{r-1}\alpha^{n(k-n)}\sum_{i=1}^r
\delta^{i-1}=$$
$$-(n-k)n^{r-1}\delta^{-1}\alpha^{n(k-n)}\sum_{i=1}^r \delta^{i}=
0$$ if $k \ne 1$. On the other hand, if $k=1$ then $\delta=1$ and
$q_{r,1}(\alpha)=-r(n-1)n^{r-1}\gamma$. It follows that
$$Q_r(\alpha)=-(0,r(n-1)n^{r-1}\gamma, 0, \dots ,0)=-r(n-1)n^{r-1}\gamma \cdot Q_1(0)\in \C^n.$$
This implies that $Z^0(r+1,2;r,1)$ does {\sl not} contain
$g(x)=x^n$.
\end{ex}

\begin{rem}
\label{rplus2}
Suppose that $r >2$ and $n=r+1$.  Example \ref{rplus1} tells us
 that if $\alpha$ and $\beta$ are
 primitive $r^2$th roots of unity then $Q_r(\alpha)$ and $Q_r(\beta)$
generate the same line $\C \cdot (0,1, \dots , 0)$ in $\C^n$.
 Since $r>2$, the number $\varphi(r^2)$ of primitive $r^2$th
roots of unity is strictly greater than $r$. (Here $\varphi$ is the
Euler function.) In particular, we may choose such $\alpha$ and
$\beta$ from different orbits (of length $r$) of the map $z \mapsto
z^n$. It follows that $Z^0(r+1,2;r,r)$ does {\sl not} contain
$g(x)=x^n$.
\end{rem}

\section{Linear independence}
\label{ind}
As was already pointed out,
Theorem \ref{mainT} is an immediate corollary of the following statement.

\begin{thm}
\label{bign}
 Let   $\ell$ be a positive integer. Let $\{r_1, \dots r_{\ell}\}$ be a sequence of $\ell$ positive integers.
  Let $\{\alpha_1, \dots, \alpha_{\ell}\}$ be a sequence of distinct complex numbers such that
  $$\alpha_i^{n^{r_i}-1}=1 \ \forall i=1, \dots, \ell.$$
  Assume that $\{\alpha_1, \dots, \alpha_{\ell}\}$
 belong to different orbits of the map $z \mapsto z^n$. Then:

  \begin{itemize}
  \item[(i)]
  the set of $\ell$ vectors
 $\{Q_{r_1}(\alpha_1), \dots , Q_{r_{\ell}}(\alpha_{\ell})\}$ in $\C^n$ is linearly independent if
 $n \ge \sum_{i=1}^{\ell}  d(\alpha_i)$. In particular, if $\sum_{i=1}^{\ell} r_i \le n$ then the $\ell$-tuple
 $\{Q_{r_1}(\alpha_1), \dots , Q_{r_{\ell}}(\alpha_{\ell})\}$  is linearly independent in $\C^n$.
 \item[(ii)]
 If $n \ge 2+ \sum_{i=1}^{\ell}  d(\alpha_i)$ then the $(\ell+1)$-tuple
 $\{Q_1(0); Q_{r_1}(\alpha_1), \dots , Q_{r_{\ell}}(\alpha_{\ell})\}$  is a linearly independent set in $\C^n$.
In particular, if $\sum_{i=1}^{\ell} r_i < n-1$ then the
$(\ell+1)$-tuple $\{Q_1(0); Q_{r_1}(\alpha_1), \dots ,
Q_{r_{\ell}}(\alpha_{\ell})\}$ is a linearly independent set in
$\C^n$.
 \end{itemize}
 \end{thm}

\begin{proof}[Proof of Theorem \ref{bign}]

In the course of the proof we will use the following elementary statement that will be proven at the end of this section.

\begin{lem}
\label{newton}
Let $d$ be a positive integer,  $S$ a set of $d$ nonzero complex numbers.
Let $c:S \to \C$ be a function such that for all positive integers $u=1, \dots ,d$
 $$\sum_{\beta \in S}\frac{c(\beta)}{\beta^u}=0.$$
Then $c(\beta)=0$ for all $\beta\in S$.

\end{lem}

Let us continue to prove Theorem \ref{bign}. Replacing each $r_i$ by
$d(\alpha_i)$ we may and will assume that $r_i=d(\alpha_i)$, i.e.,
the orbit $\Oc(\alpha_i)$ of $\alpha_i$ consists of $r_i$ distinct
elements (for all $i$). We also assume that $n \ge
\sum_{i=1}^{\ell}r_i$.

Let $\{c_1,  \dots c_{\ell}\}$ be a sequence of $\ell$ complex numbers  such that
 $$\sum_{i=1}^{\ell} c_i Q_{r_i}(\alpha_i)=0.$$
 Let $S\subset \C$ be the (disjoint) union of all $\Oc(\alpha_i)$, which consists of $\left(\sum_{i=1}^{\ell} r_i\right)$ elements. Let us define a complex valued function
 $c$ on $S$ that assigns to $\alpha \in \Oc(\alpha_i)$ the complex number
 $$c(\alpha):=n^{r_i-1}c_i.$$
 Then we obtain for all $k=0,1, \dots , n-1$
 $$0=\sum_{i=1}^{\ell} c_i q_{r_i,k}(\alpha_i)=-(n-k)\sum_{\alpha\in S}c(\alpha) (1/\alpha)^{n-k}.$$
 This implies that
 $$\sum_{\alpha\in S}c(\alpha) (1/\alpha)^{u}=0$$
 for all positive integers  $u=1, \dots, n$.

 It follows from Lemma \ref{newton} applied to $d=\sum_{i=1}^{\ell}r_i$  that
 all $c(\alpha)=0$. Since all $n^{r_i-1}\ne 0$, we conclude that all $c_i=0$. This proves (i).

 Now assume that $n \ge 2+\sum_{i=1}^{\ell}r_i$. We are going to prove (ii).
 Let $\{c_0, c_1,  \dots c_{\ell}\}$ be a sequence of $(\ell+1)$ complex numbers  such that
 $$c_0 Q_1(0)+\sum_{i=1}^{\ell} c_i Q_{r_i}(\alpha_i)=0.$$
 We have
 $$ -c_0 Q_1(0)=\sum_{i=1}^{\ell} c_i Q_{r_i}(\alpha_i).$$
 Recall that all the coordinates of $Q_1(0)$ except the second one do vanish. This implies that
 $$0=\sum_{i=1}^{\ell} c_i q_{r_i,k}(\alpha_i)=-(n-k)\sum_{\alpha\in S}c(\alpha) (1/\alpha)^{n-k}$$ for all $k=0, \dots n-1$ except $k=1$. It
 follows
 that
 $$\sum_{\alpha\in S} c(\alpha) (1/\alpha)^{u}=0$$
 for all positive integers integers $u=1, \dots, n-2$. Since $n-2\ge d$, the same arguments with Lemma \ref{newton} as above prove that $c_i=0$ for all positive integers $i$ and therefore
 $-c_0 Q_1(0)=0$, i.e., $c_0=0$.
\end{proof}

\begin{proof}[Proof of Lemma \ref{newton}]
This Lemma is a variant of well-known classical results (e.g., see \cite{Yomdin}).
Let us consider the rational function
$$X(t)=\sum_{\beta \in S} \frac{c(\beta)}{\beta - t}=
\sum_{\beta \in S} \frac{c(\beta)/\beta}{1-\frac{t}{\beta}}.$$
Clearly,
$$X(t)=\frac{Q(t)}{\prod_{\beta\in S}(\beta- t)}$$
where $Q(t)$ is a polynomial, whose degree does not exceed $d-1$. (Recall that $d=\#(S)$.)
 For each positive integer $u$, the number $\sum_{\beta\in S} c(\beta)/ \beta^u$ is
 the $(u-1)$th coefficient of the Taylor power series of $X(t)$ at the origin
  (see \cite[Ch. 1, Sect. 2]{Mc}). It follows that $X(t)$ has a zero of order $\ge d$ at the origin.
This implies that $Q(t)$ is divisible by $t^{d}$ and therefore $Q(t)=0$, i.e. $X(t)=0$.
Since $-c(\beta)$ is the residue of $X(t)$ at $t=\beta$ for all $\beta \in S$, we conclude that $c(\beta)=0$.
\end{proof}

\begin{rem}
 One may give even more elementary proof of Lemma \ref{newton},
using the nondegeneracy of the Vandermonde matrix of size $d \times d$ for {\sl distinct} numbers
$\{1/\beta \mid \beta \in S\}$.
\end{rem}

\section{Openness in Zariski topology}
\label{global}


The aim of this Section is to prove Theorem \ref{open}.
We will need the following well known  easy statement.

\begin{lem}
\label{div}
 Let $n \ge 2$ be an integer and $g(x)\in \C[x]$ is a monic degree $n$ polynomial.
Suppose that $g(x)-x$ has a multiple root say, $\alpha$. Then for all positive integers $r$ the complex number
$\alpha$ is a multiple root of $g^{\circ r}(x)-x$.
\end{lem}

\begin{proof}
 We have
$$g(\alpha)=\alpha, \ g^{\prime}(\alpha)=1.$$
It follows easily that
$$g^{\circ r}(\alpha)=\alpha, \ [g^{\circ r}]^{\prime}(\alpha)=1.$$
This means that
$$[g^{\circ r}(x)-x](\alpha)=0, \ [g^{\circ r}(x)-x]^{\prime}(\alpha)=0.$$
In other words, $\alpha$ is a multiple root of $g^{\circ r}(x)-x$.
\end{proof}

\begin{cor}
\label{divcor} Let $m$ be a positive integer that divides $r$.
Suppose that $g^{\circ m}(x)-x$ has a multiple root say, $\alpha$.
Then $\alpha$ is a multiple root of $g^{\circ r}(x)-x$.
\end{cor}

\begin{sect}
\label{numberoforbits}
 Let $g(x)\in Z_{n,r}$. If $m$ is a positive
divisor of $r$ then Corollary \ref{divcor} implies that $g(x) \in
Z_{n,m}$. The number of periodic points of exact period $m$ for
$z\mapsto g(z)$ is
$$\nu_n(m)=\sum_{m^{\prime}\mid m} \mu\left(\frac{m}{m^{\prime}}\right)n^{m^{\prime}}$$
where $\mu$ is the M\"obius function \cite[pp. 74--75]{Milnor}. In particular, the number of orbits of length $m$ equals
$$d(n, m)=\frac{\nu_n(m)}{m}$$
and therefore $\nu_n(m)$ is divisible by $m$. The explicit formula
for $\nu_n(m)$ implies that $\nu_n(m)$ is also divisible by $n$
(ibid). It follows that $\nu_n(m)$ is divisible by $nm/(n,m)$ where
$(n,m)$ is the greatest common divisor  of $n$ and $m$. On the other
hand, the number $\nu_n(m)$ is always positive. Indeed, $\nu_n(1)=n$
and for $1<m\le 5$ we have $\nu_n(m)=n^m-n^{m/p}>0$ where $p$ is the
only prime divisor of $m$. Now assume that $m>5$, i.e, $m \ge 6$.
Notice that the points of exact period $m$ are exactly the roots of
$g^{\circ m}(x)-x$ that are not roots of $g^{\circ (m/p)}(x)-x$ for
any prime divisor $p$ of $m$. Since the number of prime divisors of
$m$ does not exceed $\log_2(m)$,
$$\nu_n(m) \ge n^{m} - n^{m/2}\log_2(m).$$
It is easy to check that (under our assumptions on $n$ and $m$) we
have $n^{m/2}\ge 2^{m/2}>\log_2(m)$ and therefore $\nu_n(m)$ is
positive.

Since  $\nu_n(m)$ is divisible by $\frac{nm}{(n,m)}$,
$$\nu_n(m) \ge \frac{nm}{(n,m)}, \ d(n,m) \ge \frac{n}{(n,m)} \ge
\frac{n}{m}.$$
\end{sect}

\begin{sect}
For each positive divisor $m$ of $r$ we pick a $d(n, m)$-element set $S_m$ and consider the corresponding
 $d(n, m)$-dimensional coordinate space $\C^{S_m}$ of all $\C$-valued functions on $S_m$.

 Let us consider the Zariski-closed subset
$$\hat{Z}_{n,r} \subset Z_{n,r} \times \prod_{m\mid r}\C^{S_m}$$
that is cut out by the following equations imposed on
$$\{g; \phi_m:S_m \to \C, \ m\mid r \}\in Z_{n,r} \times \prod_{m\mid r}\C^{S_m}.$$

For each $s \in S_m$ the complex number $\phi_m(s)$ is a periodic
point, whose period divides  $m$, with respect to  $z \mapsto g(z)$,
i.e. $g^{\circ m}(\phi_m(s))=\phi_m(s)$. In addition, we require
that
$$g^{\circ r}(x)-x=
\prod_{m\mid r} \left(\prod_{s \in S_m} (x-\phi_m(s))
\prod_{i=1}^{m-1} \left(x-g^{\circ i}(\phi_m(s))\right)\right).$$ In
other words, the coefficients of both polynomials in $x$ do
coincide. Notice that
$$g^{\circ r}(\phi_m(s))-\phi_m(s)=0$$
on $\hat{Z}_{n,r}$. In particular, all the coordinate functions
$\phi_m(s)$ on $\hat{Z}_{n,r}$ are integral over the polynomial ring
$\C[\A^n]=\C[a_0, \dots , a_{n-1}]$, which is generated by the
coefficients of $g(x)=x^n + \sum_{i=0}^{n-1} a_i x^i$.

Recall that $g^{\circ r}(x)-x$ has no multiple roots. It follows
that each map $\phi_m:S_m \to \C$ is injective, its image consists
of elements of exact period $m$ while distinct elements of $S_m$ go
under $\phi_m$ to distinct orbits of length $m$; in addition, every
orbit of length $m$ contains exactly one element of $\phi_m(S_m)$.
On the other hand, for any choice of an element $\zeta$ in every
orbit of length $m$ (for each divisor $m$ of $r$) there is (exactly
one) point of $\hat{Z}_{n,r}$ that lies ``above" $g(x)$ and such
that the corresponding $\phi_m(S_m)$ consists of these $\zeta$.

By construction, the projection map of affine varieties
$\hat{Z}_{n,r} \to Z_{n,r}$ is surjective. In addition, this map
 is finite, because the ring of regular functions
 $\C[\hat{Z}_{n,r}]$ on $\hat{Z}_{n,r}$ is generated over
 $\C[Z_{n,r}]\supset \C[\A^n]=\C[a_0, \dots , a_{n-1}]$ by the coordinate functions
$\phi_m(s)$ that are integral over $\C[\A^n]$.

 Now one may ``lift'' $\Mult^r=\Mult_{r,\beta}$ to globally defined
functions on $\hat{Z}_{n,r}$. Namely, for each $s\in S_r$ the
function
$$\overline{\Mult}_{r,s}:\hat{Z}_{n,r} \to \C, \{g; \phi_m:S_m \to \C, \ m\mid r \}\mapsto \Mult_{r,\phi_r(s)}(g)$$
is a globally defined regular function.  If $\phi_m(s)$ and $\beta$
lie in the same orbit of length $m$ then this function coincides
with the composition of projection map $\hat{Z}_{n,r} \to Z_{n,r}$
and $\Mult_{r,\beta}$.
 It is also clear that the vector function
$$\overline{\grad}(\Mult_{r,s}) :\hat{Z}_{n,r} \to \C^n, \ \{g; \phi_m:S_m \to \C^n, \ m\mid r \}\mapsto
\grad(\Mult_{r,\phi_r(s)})_{\mid g(x)}$$ is a regular map that
coincides with the composition of projection map $\hat{Z}_{n,r} \to
Z_{n,r}$  and $\grad(\Mult_{r,\beta}): Z_{n,r}\to \C^n$  with
$\beta=\phi_r(s)$.

Let $\ell$ be a positive integer that does not exceed $d(n,r)$. If
$D$ is an $\ell$-element subset of $S_r$ let us consider the the
subset $X_D$ of points $v \in \hat{Z}_{n,r}$ such that the
collection of $\ell$ vectors $\{\overline{\grad}(\Mult_{r,s})(v)\mid
s\in D\}$ is linearly dependent in $\C^n$. Clearly, $X_D$ is a
Zariski-closed subset in $\hat{Z}_{n,r}$. It follows that the union
$X$ of all $X_D$ (where $D$ runs through all $\ell$-element subsets
of $S_r$) is also closed in  $\hat{Z}_{n,r}$. The finiteness of the
projection map implies that the image $\bar{X}$ of $X$ in $Z_{n,r}$
is also Zariski-closed (\cite[Ch. 1, Sect. 5.3]{Sh}). On the other
hand, one may easily check that $\bar{X}$ is the complement of
$Z^0(n,\ell; r, r, \dots ,r)$ in $Z(n,\ell; r,r \dots, r)=Z_{n,r}$.
It follows that $Z^0(n,\ell; r,r \dots, r)$ is Zariski-open in
$Z_{n,r}$. This proves Theorem \ref{open} in the case when $r_1=r_2=
\dots = r_{\ell}$.
\end{sect}

\begin{sect}
Now let us consider the general case. Let $d$ be the number of distinct elements in
 the sequence $\{r_1, \dots ,r_{\ell}\}$ and
$R$ the corresponding $d$-element set of positive integers. For each
$r\in R$ we denote by $l(r)$ the number of $i$ with $r_i=r$. If
there is $r$ with $l(r)> d(n,r)$ then $Z(n,\ell; r_1, r_2, \dots
,r_{\ell})$ is empty. So further we assume that $l(r)\le d(n,r)$ for
all $r\in R$.

We write $\hat{Z}_{n,r}^{\prime}$ for the preimage of $Z(n,\ell;
r_1, r_2, \dots ,r_{\ell})$ in $\hat{Z}_{n,r}$. The natural
 regular map
$$\hat{Z}_{n,r}^{\prime}\to Z(n,\ell; r_1, r_2, \dots ,r_{\ell})$$
  is finite, because the base change of a finite map is finite \cite[Sect. 8, Prop. 8.22]{Milne}. Let us consider the fiber product
$\hat{Z}_{n,R}$ of all $\hat{Z}_{n,r}^{\prime}$ (where $r$ runs
through $R$) over $Z(n,\ell; r_1, r_2, \dots ,r_{\ell})$. If we
write $\pi$ for the natural regular map $$\hat{Z}_{n,R}\to Z(n,\ell;
r_1, r_2, \dots ,r_{\ell})$$ then $\pi$ is finite, because the
composition  of finite maps is also finite \cite[Sect. 8, Prop. 8.4
and 8.22]{Milne}. We denote by $\pi_r$
 the natural finite regular map
$$\pi_r:\hat{Z}_{n,R} \to \hat{Z}_{n,r}^{\prime}.$$
Now for each $r \in R$ pick a $l(r)$-element subset $D_r$ in $S_r$
and consider the subset $X_{\{D_r\mid r\in R\}}$ of points $v \in
\hat{Z}_{n,R}$ such that the collection of $\left(\sum_{r\in
R}l(r)\right)$ vectors
$$\{\overline{\grad}(\Mult_{r,s_r})(v_r)\mid s_r\in D_r, \ r \in R\}$$
 is linearly dependent in $\C^n$.
Here $v_r=\pi_r(v)\in \hat{Z}_{n,r}^{\prime}$. Clearly,
$X_{\{D_r\mid r\in R\}}$ is a closed algebraic subvariety in
$\hat{Z}_{n,R}$. Now let us consider the union $Y$ of all such
$X_{\{D_r\mid r\in R\}}$ for all choices of $\{D_r\mid r\in R\}$. It
is also clear that $Y$ is also a closed algebraic subvariety in
$\hat{Z}_{n,R}$. Since  $\pi$ is finite,
  the image $\pi(Y)$ is a closed algebraic subvariety in
$Z(n,\ell; r_1, r_2, \dots ,r_{\ell})$. On the other hand, one may easily check
that $Z^{0}(n,\ell; r_1, r_2, \dots ,r_{\ell})$ is the complement of $\pi(Y)$ in $Z(n,\ell; r_1, r_2, \dots ,r_{\ell})$.
It follows that $Z^{0}(n,\ell; r_1, r_2, \dots ,r_{\ell})$ is Zariski-open in $Z(n,\ell; r_1, r_2, \dots ,r_{\ell})$.
This ends the proof of Theorem \ref{open}.
\end{sect}

{\bf Acknowledgements.} I am grateful to Yulij S. Ilyashenko,
Vladimir L. Popov and Victor S. Kulikov  for  stimulating
discussions. My special thanks go to Tatiana Bandman, whose comments
helped to improve the exposition.

 Part of this work was done during my stay at
Weizmann Institute of Science Department of Mathematics in May--June
of 2012 and at Centre Interfacultaire Bernoulli (\'Ecole
Polytechnique F\'ed\'erale de Lausanne) in July--August of 2012:
 I am grateful to both of them  for the hospitality.

 \end{document}